\documentclass[letterpaper]{article}

\usepackage[english]{babel}
\usepackage[utf8x]{inputenc}
\usepackage[T1]{fontenc}
\usepackage[affil-it]{authblk}

\usepackage[letterpaper,top=3cm,bottom=3cm,left=3cm,right=3cm,marginparwidth=1.75cm]{geometry}

\usepackage{amsmath}
\usepackage{amsthm}
\usepackage{amssymb}
\usepackage[all,cmtip]{xy}
\usepackage{graphicx}
\usepackage[colorinlistoftodos]{todonotes}
\usepackage[colorlinks=true, allcolors=blue]{hyperref}
\usepackage{tikz-cd}
\usepackage{theoremref}

\theoremstyle{plain}
\newtheorem{theorem}{Theorem}[section]
\newtheorem{lemma}[theorem]{Lemma}
\newtheorem{prop}[theorem]{Proposition}
\newtheorem{cor}[theorem]{Corollary}

\theoremstyle{definition}
\newtheorem{exx}[theorem]{Example}
\newtheorem{deff}[theorem]{Definition}

\newcommand{\F}{\mathbb{F}}
\newcommand{\Q}{\mathbb{Q}}
\newcommand{\C}{\mathbb{C}}
\newcommand{\FR}{\mathbb{F}R_{\mathbb{F}}}
\newcommand{\FRG}{\mathbb{F}R_{\mathbb{F},G}}

\newcommand{\CR}{\mathbb{C}R_{\mathbb{C}}}
\newcommand{\CRG}{\mathbb{C}R_{\mathbb{C},G}}
\newcommand{\PKQ}{\mathcal{P}_{kR_{\mathbb{Q}}}}

\newcommand{\PKQG}{\mathcal{P}_{kR_{\mathbb{Q},G}}}
\newcommand{\E}{\mathbb{E}}
\newcommand{\Z}{\mathbb{Z}}

\newcommand{\RF}{kR_{\mathbb{F}}}
\newcommand{\RQ}{kR_{\mathbb{Q}}}
\newcommand{\RQG}{k R_{\mathbb{Q},G}}
\newcommand{\RFG}{kR_{\mathbb{F},G}}
\newcommand{\ER}{\widehat{kR_{\mathbb{Q}}}}
\newcommand{\ERG}{\widehat{kR_{\mathbb{Q},G}}}
\newcommand{\ERFG}{\widehat{kR_{\mathbb{F},G}}}

\title{On the ideals and essential algebras of shifted functors of linear representations.}
\author{Benjam\'in Garc\'ia}
\affil{Centro de Ciencias Matem\'aticas, UNAM campus Morelia\\ e-mail address: benjamingarcia@matmor.unam.mx}

\begin{document}
\maketitle

\begin{abstract}
We present a study on the shifted Green biset functors $\RFG$ of linear $\F$-representations with coefficients over $k$, for fields $k$ and $\F$ of characteristic zero and a finite group $G$. We provide a criterion for the vanishing of their essential algebras and we prove that the condition of uniqueness of minimal groups for simple modules holds for these functors. We give a parametrization of a family of simple $\RQG$-modules. We also prove the semisimplicity of the category of $\CRG$-modules by means of an equivalence with a category of modules over a semisimple algebra.\\

\end{abstract}

\section{Introduction.}

A Green biset functor $A$ is a biset functor over a unitary commutative ring $k$, together with a family of bilinear products. There is a natural way to define ideals and modules over $A$, and the study and classification of these functors become an interesting problem as it takes part on the wider problem of classification of $k$-linear functors over $k$-linear categories. Green biset functors arise naturally from representation theory of finite groups, which at the same time provides tools for the understanding of their internal structure and their modules. Many of these functors have been extensively studied, and for some of them, families of simple modules and ideals have been parametrized by means of the most diverse methods (see, for example: Barker \cite{LB2,LB}, Bouc \cite{SBb2}, Ducellier \cite{MD1} \& Romero \cite{ROM2,ROM1}).

For a Green biset functor $A$, a known result due to Nadia Romero \cite{ROM1} gives a parametrization of its simple modules by means of isomorphism classes of seeds $(H,V)$ consisting of a group $H$ such that the essential algebra $\widehat{A}(H)$ is non-zero, and a simple $\widehat{A}(H)$-module $V$, under the assumption that minimal groups for simple modules are unique up to group isomorphism. Romero pointed out in \cite{ROM2} that this condition does not hold in general: the monomial Burnside functor $kB^1_{C_4}$ has a simple module for which there are two non-isomorphic minimal groups. However, even if we know that the condition of uniqueness holds for $A$, a deeper understanding of its simple modules requires to go further in the study of its essential algebras and their modules. On the one hand, we would like to find necessary and sufficient conditions on a group $H$ for $\widehat{A}(H)$ to be zero; on the other hand, if we know that $\widehat{A}(H)$ is not zero, we would like to know all of its simple modules. Both problems are in general very complicated, and for most Green biset functors, the best we can expect are partial solutions.

If $k$ is a field of characteristic zero, the functor $\RQ$ and its modules, known as rhetorical biset functors, were studied by Laurence Barker \cite{LB}. Barker proved that the condition of uniqueness on minimal groups for simple modules holds for $\RQ$, and simple rhetorical biset functors are parametrized by isomorphism classes of seeds $(H,V)$ where $H$ is a cyclic group and $V$ is a primitive $k$Out$(H)$-module. 

This article is devoted to the study of the shifted functors $\RFG$, its essential algebras and its ideals, for fields $k$ and $\F$ of characteristic zero and $G$ a finite group. These functors are Green biset functors and projective $\RF$-modules at the same time, and making $G$ run over a set of representatives of the isomorphisms classes of finite groups, we get a set of projective generators of the category $\RF-\mathcal{M}od$ of $\RF$-modules. 

In Section 3, we introduce $(\F,G)$-rhetorical biset functors as modules over $\RFG$, and $G$-rhetorical biset functors as modules over $\RQG$; these concepts generalize rhetorical biset functors. We give a characterization of the lattice of ideals of $\RFG$ by means of an isomorphism of lattices between the set of ideals $\mathfrak{I}_{k,\F,G}$ of $\RFG$ and the power set of the set $\Omega(k,\F,G)$, defined to be the orbit space of an action of $Gal(k[\omega],k)$ on the set $c_{\F}(G)$ of $\F$-conjugacy classes of $G$, where $\omega$ is an $n$-th primitive root of $1$ and $n$ is the exponent of $G$. This implies that $\RFG$ is a semisimple $(\F,G)$-rhetorical biset functor. Then we prove that the essential algebras of $\RFG$ vanish for non-cyclic groups, and that if $H$ is a cyclic group such that $\ERFG(H)$ is non-zero, then $\ER(H)$ is non-zero. This implies that the condition of uniqueness of minimal groups holds for $\RFG$. Finally, we provide a parametrization of the family of simple $G$-rhetorical biset functors whose minimal groups have order relatively prime to the order of $G$.

In Section 4, we focus on the case $k=\F=\C$. Romero proved in \cite{ROM1} that the functor $\CR$ is a simple $\CR$-module, by proving that in $\mathcal{P}_{\CR}$, the hom-sets are generated by morphisms which factor through the trivial group. We prove that this property holds in the shifted case, implying that the trivial group is minimal for any $\CRG$-module. Then it is natural to think that any $\CRG$-module is completely determined by its value at the trivial group. This is actually so, and we provide an equivalence between the category of $\CRG$-modules and $\CR(G)-Mod$, implying among other things the semisimplicity of the first one.

\section{Preliminaries.}

Let $G$ and $H$ be finite groups. An \textit{$(H,G)$-biset} $X$ is a set with a left action by $H$ and a right action by $G$ such that these actions commute. We will often write \textit{$_HX_G$ is a biset} to say that $X$ is an $(H,G)$-biset The opposite biset $X^{op}$ is the $(G,H)$-biset which as set is $X$ but with actions defined by $g\cdot x \cdot h=h^{-1}xg^{-1}$ for $x\in X^{op}$, $g\in G$ and $h\in H$. If $\phi: H\longrightarrow G$ is a group homomorphism, then $G$ becomes an $(H,G)$-biset that we denote by $_{H^{\phi}}G_G$, with actions defined by $h\cdot g \cdot g':=\phi(h)gg'$ for $h\in H$, $g\in {_{H^{\phi}}G_G}$ and $g'\in G$, while $_GG_{^{\phi}H}$ stands for its opposite. Important cases of this construction are the following: if $i:H\longrightarrow G$ is an inclusion, then the \textit{induction from $H$ to $G$} is the biset Ind$^G_H={_GG_{^iH}}$ and the \textit{restriction from $G$ to $H$} is Res$^G_H={_{H^i}G_G}$; if $N\unlhd G$ and $\pi:G\longrightarrow G/N$ is the canonical projection, then the \textit{deflation from $G$ to $G/N$} is Def$_{G/N}^G={_{G/N}G/N_{^{\pi}G}}$ and the \textit{inflation from $G/N$ to $G$} is Inf$^G_{G/N}=_{{G^{\pi}}{G/N}_{G/N}}$; if $\phi:H\longrightarrow G$ is a group isomorphism, $_GG_{^{\phi}H}$ is often denoted by Iso$(\phi)$. These bisets are often known as \textit{basic biset operations}.

We write $B(H,G)$ for the Burnside ring of the category of finite $(H,G)$-bisets. If $K$ is another finite group and $_KY_H$ is a biset, there is a left action by $H$ in $Y\times X$ given by $h\cdot (y,x)=(yh^{-1},hx)$, for $h\in H$ and $(y,x)\in Y\times X$. We define the \textit{composition of $Y$ and $X$}, denoted by $Y\circ X$ or $Y\times_HX$, as the $(K,G)$-biset which as set is the orbit space of the action defined above with $k\cdot [y,x]\cdot g=[ky,xg]$ for $k\in K$, $g\in G$ and where $[y,x]$ denotes the class of $(y,x)$ in $Y\circ X$. This operation is easily checked to be biadditive and associative up to isomorphism, and it holds that $H\circ X\cong X\cong X\circ G$ for any biset $_HX_G$, so it induces a biadditive application $B(K,H)\times B(H,G)\longrightarrow B(K,G)$ that we call \textit{composition}. If $_HX_G$ and $_LY_K$ are bisets, $X\times Y$ becomes naturally a $(H\times L,G\times K)$-biset; this induces a biadditive aplication $B(H,G)\times B(L,K)\longrightarrow B(H\times L, G\times K)$, and we denote by $\alpha \times \beta$ the image of $(\alpha,\beta)\in B(H,G)\times B(L,K)$ under this map.

The \textit{biset category} $\mathcal{C}$ has all finite groups as object class and hom-sets given by $\mathcal{C}(G,H)=B(H,G)$ for any pair of finite groups $G$ and $H$, with composition induced by the composition of bisets and identity for $G$ given by the isomorphism class of the $(G,G)$-set $G$. The biset category is a preadditive category. If $k$ is a commutative ring with unit and $\mathcal{D}$ is a preadditive subcategory of $\mathcal{C}$, we can consider the $k$-linearization $k\mathcal{D}$ of $\mathcal{D}$ which is the category whose objects are the same of $\mathcal{D}$ and hom-sets given by $k\mathcal{D}(G,H)=k\otimes \mathcal{D}(G,H)$ for any pair of finite groups $G$ and $H$ with composition induced from that of $\mathcal{D}$ by $k$-linear extension.

A \textit{biset functor for $\mathcal{D}$ over $k$} is a $k$-linear functor from $k\mathcal{D}$ to $k-Mod$. We denote by $\mathcal{F}_{\mathcal{D},k}$ the category of biset functors for $\mathcal{D}$ over $k$ with morphisms given by natural transformations. By now we are only interested in the case $\mathcal{D}=\mathcal{C}$, so for a biset functor we will always mean a biset functor for $\mathcal{C}$ over $k$, without danger of confusion. If $F$ is a biset functor, a finite group $H$ is said to be \textit{minimal for $F$} if $F(H)\neq 0$ and $F(K)=0$ for any finite group $K$ such that $|K|<|H|$.

The Burnside functor $kB$ and the \textit{functor of linear representations} $\RF$ for a field $\F$ of characteristic zero are biset functors. We recall the definition of $\RF$. For any finite group $G$, the \textit{ring of linear representations of $G$}, denoted by $R_{\F}(G)$, is the Grothendieck group of the category of finitely generated $\F G$-modules with product induced by tensor product over $\F$, and  for any finite biset $_HX_G$, the map sending the isomorphism class of a finitely generated $\F G$-module $M$ to the isomorphism class of the $\F H$-module $\F X \otimes_{\F G} M$ induces a group homomorphism $R_{\F}(X):R_{\F}(G)\longrightarrow R_{\F}(H)$, thus $R_{\F}$ is a biset functor over $\Z$. If $M$ is a finitely generated $\F G$-module, then the \textit{$\F$-character afforded by $M$} is the function $\chi_M:G\longrightarrow \F$ defined on an element $g$ of $G$ as the trace of $\psi(g)$, where $\psi: G\longrightarrow GL_{\F}(M)$ is the linear representation of $G$ on $M$, and the \textit{ring of $\F$-characters} of $G$ is the abelian group of all the $\F$-valuated functions on $G$ which can be expressed as the difference of characters afforded by modules, with product defined pointwise. Since $\F$ has characteristic zero, then $R_{\F}(G)$ and the ring of $\F$-characters on $G$ are isomorphic, and we often identify these rings. Then $\RF(G)$ is defined as $k\otimes R_{\F}(G)$, while $\RF(X):\RF(G) \longrightarrow \RF(H)$ is the $k$-linear extension of $R_{\F}(X)$. 

\subsection{Modules over Green biset functors.}

Now we come to the definition of Green biset functor, as it appears in Bouc \cite[Chapter 8]{SBb2}. Recall that the category $grp$ of all finite groups is a monoidal category with respect to the cartesian product, where the association arrows are given by the canonical isomorphisms $\alpha_{L,K,H}: L\times (K\times H)\longrightarrow (L\times K)\times H$, and the left and right unit arrows are given by the canonical isomorphisms $\lambda_H:1\times H\longrightarrow H$ and $\rho_H: H\times 1\longrightarrow H$ respectively, for any finite groups $L$, $K$ and $H$.

\begin{deff}
A \textit{Green biset functor $A$} is a biset functor together with bilinear \textit{products} 
$$A(K)\times A(H)\longrightarrow A(K \times H)$$
denoted by $(a,b)\mapsto a\times b$, for any pair of finite groups $K$ and $H$, satisfying the following conditions:
\begin{enumerate}
\item (Associativity) Let $H$, $K$ and $L$ be finite groups. Then 
$$A(Iso(\alpha_{L,K,H}))(a\times (b\times c)) = (a\times b)\times c$$
for any $a\in A(L)$, $b\in A(K)$ and $c\in A(H)$.
\item (Identity element) There exists an \textit{identity element} $\epsilon \in A(1)$ such that
$$A(Iso(\lambda_H))(\epsilon \times a)=a=A(Iso(\rho_H))(a \times \epsilon)$$
for any finite group $H$ and any $a \in A(H)$.
\item (Functoriality) Let $H$, $K$, $L$ and $T$ be finite groups. If $\alpha \in kB(T,K)$ and $\beta \in kB(L,H)$, then
$$A(\alpha)(a) \times A(\beta )(b)=A(\alpha \times \beta)(a\times b)$$
for any $a\in A(K)$ and $b\in A(H)$. 
\end{enumerate}

If $C$ is another Green biset functor, a \textit{morphism of Green biset functors} $f:A\rightarrow C$ is a morphism  of biset functors such that $f_{K\times H}(a \times b)= f_K(a) \times f_H(b)$ for any finite groups $K$ and $H$ and any $a \in A(K)$ and $b \in A(H)$, and $f_1(\epsilon)=\epsilon$.
\end{deff}

There are many other equivalent definitions of a Green biset functor. For most of our purposes in this paper, this definition is easier to handle; however, the following proposition gives an equivalent definition that will be useful in the following section.

\begin{prop}[Romero {{\cite[4.2.3]{ROMthesis}}}]
Let $A$ be a biset functor. Then the following statements are equivalent:
\begin{enumerate}
\item A is a Green biset functor.
\item For any finite group $H$, $A(H)$ is an associative $k$-algebra with unit, and for any homomorphism of finite groups $\psi: H\longrightarrow K$, the following statements hold:
\begin{enumerate}
\item $A(_{H^{\psi}}K_K): A(K)\longrightarrow A(H)$ is a homomorphism of unitary $k$-algebras. 
\item (Frobenius identities) If $a \in A(H)$ and $b\in A(K)$, then
$$A(_KK_{^{\psi}H})(a)b=A(_KK_{^{\psi}H})(aA(_{H^{\psi}}K_K)(b)),$$
$$bA(_KK_{^{\psi}H})(a)=A(_KK_{^{\psi}H})(A(_{H^{\psi}}K_K)(b)a).$$
\end{enumerate}
\end{enumerate}
\end{prop}

For details on how (1) gives rise to (2) and vice versa, see Barker \cite[4.1 \& 4.4]{LB2}. Examples of Green biset functors are the Burnside functor $kB$ and the functor of linear representations $\RF$ for a field $\F$ of characteristic zero with bilinear products induced by the external product of modules and identity element given by the class of the trivial module. The functor $kB$ is an initial object in the category of Green biset functors.

\begin{deff}
Let $A$ be a Green biset functor. A \textit{left $A$-module} is a biset functor $M$ together with bilinear applications $\xymatrix{
{A(K)\times M(H)}\ar[r]^-{\times} &{M(K \times H)}
}$ denoted by $(a,m)\mapsto a\times m$, for any pair of finite groups $H$ and $K$, satisfying the following conditions:
\begin{enumerate}
\item (Associativity) Let $H$, $K$ and $L$ be finite groups. Then 
$$M(Iso(\alpha_{L,K,H}))(a\times (b\times m)) = (a\times b)\times m$$
for any $a\in A(L)$, $b\in A(K)$ and $m\in M(H)$.
\item (Identity element) $M(Iso(\lambda_H))(\epsilon \times m)=m$ for any finite group $H$ and any $m \in M(H)$.
\item (Functoriality) Let $H$, $K$, $L$ and $T$ be finite groups. If $\alpha \in kB(T,K)$ and $\beta \in kB(L,H)$, then
$$A(\alpha)(a) \times M(\beta )(m)=M(\alpha \times \beta)(a\times m)$$
for any $a\in A(K)$ and $m\in M(H)$. 
\end{enumerate}
\end{deff}

If $N$ is another $A$-module, a \textit{morphism of left $A$-modules} $f:M\longrightarrow N$ is a morphism of biset functors such that $f_{K\times H}(a\times m)=a\times f_H(m)$ for any finite groups $K$ and $H$, any $a\in A(K)$ and any $m\in M(H)$. In a similar way, one can define a \textit{right A-module}. 

\begin{prop}[G. {{\cite[1.6]{tesina}}}] \thlabel{AMOD}
Let $A$ be a Green biset functor and $M$ be a biset functor. Then the following statements are equivalent: 
\begin{enumerate}
\item $M$ is a left $A$-module.
\item For any finite group $H$, $M(H)$ is a left $A(H)$-module, and for any homomorphism of finite groups $\psi: H\longrightarrow K$, the following statements hold:
\begin{enumerate}
\item $M(_{H^{\psi}}K_K)(am)=A(_{H^{\psi}}K_K)(a)M(_{H^{\psi}}K_K)(m)$ for any $a\in A(K)$ and $m\in M(K)$. 
\item (Frobenius identities) If $a \in A(H)$ and $m\in M(K)$, then
$$A(_KK_{^{\psi}H})(a)m=M(_KK_{^{\psi}H})(aM(_{H^{\psi}}K_K)(m)),$$
and if $b\in A(K)$ and $t\in M(H)$, then
$$bM(_KK_{^{\psi}H})(t)=M(_KK_{^{\psi}H})(A(_{H^{\psi}}K_K)(b)t).$$
\end{enumerate}
\end{enumerate}
\end{prop}

Details on how (1) and (2) are related can be found in Barker \cite[5.8 \& 5.11]{LB2}. For any Green biset functor $A$, we denote by $A-\mathcal{M}od$ the category of all its left $A$-modules. A \textit{left ideal} $I$ of $A$ is a biset subfunctor of $A$ which is also an left $A$-module by restriction of the product of $A$. Similarly, one can define a \textit{right ideal} and a \textit{two-sided ideal} of $A$. A Green biset functor is said to be \textit{simple} if it has no proper two-sided ideal other than 0. By \thref{AMOD}, $I$ is a left (resp. right, two-sided) ideal of $A$ if and only if it is a biset subfunctor of $A$ such that $I(H)$ is a left (resp. right, two-sided) ideal of $A(H)$ for any finite group $H$. 

If $H$ is a finite group, we denote by $\overrightarrow{H}$ the $(H\times H,1)$-biset $H$ with actions given $(h_1,h_2)\cdot h \cdot 1=h_1hh^{-1}_2$, while $\overleftarrow{H}$ denotes its opposite biset. If $L$, $K$ and $H$ are finite groups, we define
$$\beta\circ \alpha=A(L\times \overleftarrow{K} \times H)(\beta\times \alpha)$$
for any $\alpha \in A(K\times H)$ and $\beta \in A(L\times K)$, inducing a bilinear map $A(L\times K) \times A(K\times H) \longrightarrow A(L\times H)$.

\begin{deff}
Let $A$ be a Green biset functor. The \textit{category associated to $A$} is defined as the category $\mathcal{P}_A$ whose objects are all finite groups and whose hom-sets are given by $\mathcal{P}_A(H,K)=A(K\times H)$ for any pair of finite groups $H$ and $K$, with the composition given by the bilinear applications defined in the previous paragraph and identity for $H$ given by $Id_H=A(\overrightarrow{H})(\epsilon)$.
\end{deff}

The category $\mathcal{P}_A$ is $k$-linear. Let $Fun_k(\mathcal{P}_A,k-Mod)$ denote the category of all $k$-linear functors from $\mathcal{P}_A$ to $k-Mod$.

\begin{prop}[Bouc {{\cite[8.6.1]{SBb2}}}]
The categories $Fun_k(\mathcal{P}_A,k-Mod)$ and $A-\mathcal{M}od$ are $k$-linearly equivalent.
\end{prop}

Details on the proof for this result can be found in Romero \cite[2.11]{ROM1}.

\subsection{The problem of classification of simple modules.}

From now, by an $A$-module we always mean a left $A$-module. The classification of simple $A$-modules is in general a difficult problem. Under suitable conditions on $A$, there are nice parametrizations of simple $A$-modules. Recall that a finite group $H$ is \textit{minimal} for an $A$-module $M$ if it is minimal in the sense of biset functors.

\begin{deff}
Let $A$ be a Green biset functor and $H$ a finite group. The \textit{essential algebra of $A$ in $H$} is defined as the quotient
$$\widehat{A}(H)=\frac{End_{\mathcal{P}_A}(H)}{I_A(H)}$$
where 
$$I_A(H)= \sum_{|K|<|H|}\left\langle\mathcal{P}_A(K,H) \circ \mathcal{P}_A(H,K) \right\rangle$$
which is a two-sided ideal of $End_{\mathcal{P}_A}(H)$. We write $\widehat{a}$ for the class in $\widehat{A}(H)$ of an element $a$ of $A(H\times H)$.
\end{deff}

The essential algebras are unitary associative algebras, but in general we cannot say whether or not they vanish for an arbitrary group $H$. However, the following may be a useful criterion: if $f:A\rightarrow C$ is a morphism of Green biset functors, then it induces a homomorphism of unitary $k$-algebras $\widehat{f}_H: \widehat{A}(H) \rightarrow \widehat{C}(H)$ defined by the rule $\widehat{f}_H(\widehat{a})=\widehat{f_{H\times H}(a)}$, for any finite group $H$ and $a\in A(H\times H)$. Then if $\widehat{A}(H)$ is zero, so is $\widehat{C}(H)$.

If $S$ is an $A$-module, then $S(H)$ is naturally an $End_{\mathcal{P}_A}(H)$-module. If $H$ is minimal for $S$, then $S(H)$ becomes naturally an $\widehat{A}(H)$-module. If $S$ is simple, then $S(H)$ is simple as $\widehat{A}(H)$-module. Now, given a group isomorphism $\psi:H\longrightarrow K$, we define a $k$-algebra isomorphism $\widehat{A}(H) \longrightarrow \widehat{A}(K)$ by the rule $\widehat{a}\mapsto \widehat{b\circ a\circ b'}$ for $a\in A(H\times H)$, where $b=\lambda_{K\times H}($Iso$(\psi))$ and $b'=\lambda_{H\times K}($Iso$(\psi^{-1}))$; we write $^{\psi}M$ for the restriction of scalars of an $\widehat{A}(K)$-module $M$ via this isomorphism.

\begin{deff}
A \textit{seed} on $\mathcal{P}_A$ is a pair $(H,V)$ consisting of a finite group $H$ such that $\widehat{A}(H)$ is not zero and a simple $\widehat{A}(H)$-module $V$. A seed $(H,V)$ is \textit{isomorphic} to a seed $(K,W)$ if there exists a group isomorphism $\psi:H\longrightarrow K$ such that $^{\psi} W$ is isomorphic to $V$ as $\widehat{A}(H)$-modules.
\end{deff}

Now recall that for any finite group $H$, the evaluation at $H$ gives an exact functor 
$$ev_H: A-\mathcal{M}od \longrightarrow End_{\mathcal{P}_A}(H)-Mod$$ 
so it has left and right adjoints $L_H$ and $R_H$ respectively. The functor $L_H:End_{\mathcal{P}_A}(H)-Mod \longrightarrow A-\mathcal{M}od$ can be defined as the functor whose value at an $End_{\mathcal{P}_A}(H)$-module $V$ is the $A$-module $L_{H,V}$, whose value at a finite group $K$ is
$$L_{H,V}(K)=A(K\times H)\otimes_{End_{\mathcal{P}_A}(H)}V$$
and 
$$L_{H,V}(\alpha): L_{H,V}(K) \longrightarrow L_{H,V}(L)$$
$$\beta \otimes v\mapsto (\alpha \circ \beta) \otimes v$$
for any arrow $K\xrightarrow{\alpha} L$ in $\mathcal{P}_A$, while for any homomorphism of $End_{\mathcal{P}_A}(H)$-modules $\psi:V\longrightarrow W$, the arrows
$$(L_{H,\psi})_K: L_{H,V}(K) \longrightarrow L_{H,W}(K)$$
$$\beta \otimes v\mapsto \beta \otimes \psi(v)$$
define a natural transformation $L_{H,\psi}:L_{H,V}\longrightarrow L_{H,W}$. If $H$ is such that $\widehat{A}(H)\neq 0$ and $V$ is a simple $\widehat{A}(H)$-module, then $V$ is a simple $End_{\mathcal{P}_A}(H)$-module by restriction of scalars. Then the $A$-module $L_{H,V}$ has a unique maximal submodule $J_{H,V}$ and a unique simple quotient
$$S_{H,V}=L_{H,V}/J_{H,V}$$
with $H$ as minimal group and $S_{H,V}(H)\cong V$ as $\widehat{A}(H)$-modules. It is not hard to see that if $(K,W)$ and $(H,V)$ are isomorphic, then $S_{K,W}\cong S_{H,V}$.

\begin{prop}[Romero {{\cite[4.2]{ROM1}}}] \thlabel{thNadia}
Let $A$ be a Green biset functor such that any simple $A$-module has a unique minimal group up to group isomorphism. There is a bijection between the set of isomorphism classes of simple $A$-modules and the set of isomorphism classes of seeds on $\mathcal{P}_A$.
\end{prop}

The condition of uniqueness on minimal groups does not hold in general, as it was exposed in Romero \cite[12]{ROM2}. We give some examples to the proposition.

\begin{exx}
The Burnside functor $kB$ is a Green biset functor, which turns out to be an initial object in the category of Green biset functors. There is an equivalence of categories between $\mathcal{F}_{\mathcal{C},k}$ and $kB-\mathcal{M}od$. Uniqueness of minimal groups holds for simple biset functors, and the classification of simple biset functors by means of isomorphism classes of seeds was given by Bouc in \cite{SB1} in 1996. The category $\mathcal{P}_A$ was introduced later by Bouc as a generalization of the biset category. More details can be found in Bouc \cite[Chapter 4]{SBb2}.
\end{exx}

\begin{exx}
If $H$ is a finite group and $\F$ is a field of characteristic zero, the linearization functor $X\mapsto \F X$ from $H-set$ to $\F H-mod$ induces a $k$-algebra homomorphism $\lambda_H: kB(H)\longrightarrow \RF(H)$, which defines a morphism of Green biset functors $\lambda: kB \longrightarrow \RF$ that we call \textit{linearization morphism}. Linearization is the only morphism of Green biset functors from $kB$ to $\RF$ since $kB$ is initial in the category of Green biset functors. The image $\Lambda=\lambda(kB)$ of $kB$ in $\RF$ is a Green biset functor, and its modules are known as \textit{rhetorical biset functors}. When $k$ is a field of characteristic zero, \thref{thNadia} applies to this functor since simple rhetorical biset functors are the simple biset functors annihilated by $ker\;\lambda$. Simple rhetorical biset functors were studied and classified by Barker in \cite{LB}.
\end{exx}

As we said before, even to know if $\widehat{A}(H)$ vanishes or not for a given $H$ can be a difficult problem. For the trivial group however, we have the following result.

\begin{lemma} \thlabel{L1}
Let $A$ be a Green biset functor. Then $\widehat{A}(1)\cong End_{\mathcal{P}_A}(1)\cong A(1)$ as $k$-algebras.
\end{lemma}
\begin{proof}
We have
$$_{1\times 1}1 \times \overleftarrow{1} \times 1_{1\times 1\times 1\times 1}\cong _{(1  \times 1)^{\Delta} }1 \times 1  \times 1\times 1_{1\times 1\times 1\times 1}$$
as bisets, so we have 
$$a\circ b=A(1 \times \overleftarrow{1} \times 1)(a\times b)=A(_{(1  \times 1)^{\Delta} }1 \times 1  \times 1\times 1_{1 \times 1  \times 1\times 1})(a\times b)=ab$$
for all $a, b\in A(1\times 1)$. Then $k$-algebras $End_{\mathcal{P}_A}(1)$ and $A(1)$ are isomorphic. Since $I_A(1)=0$, it follows that $\widehat{A}(1)\cong A(1)$ as $k$-algebras.
\end{proof}

From now, if $A$ is a Green biset functor, $H$ is a finite group, $a\in A(H)$ and $_KX_H$ is a basic biset operation, we write $Xa$ instead of $A(X)(a)$, e.g., if $H\leq K$, then Ind$^K_H a$ means $A($Ind$^K_H)(a)$. The following result is part of Barker \cite[4.4]{LB2}, but stated and proved in a slightly different way.

\begin{lemma} \thlabel{L2}
Let $A$ be a Green biset functor. Then for all finite groups $H$ and $K$
$$Inf_H^{H\times 1}a \circ Inf^{1\times K}_Kb=a\times b$$
for all $a\in A(H)$ and $b\in A(K)$.
\end{lemma}
\begin{proof}
\begin{align*}
Inf_H^{H\times 1}a \circ Inf^{1\times K}_Kb &=A(H\times \overleftarrow{1} \times K)(A(_{H\times 1}H_H)(a)\times A(_{1\times K}K_K)(b))\\
&=A(H\times \overleftarrow{1} \times K \circ \; H \times K_{H\times K})(a \times b)=a\times b
\end{align*}
since
$$H\times \overleftarrow{1} \times K \circ \; _{H\times 1 \times 1 \times K}H \times K_{H\times K} \longrightarrow H\times K$$
$$[(h_1,1,k_1),(h_2,k_2)]\mapsto (h_1h_2,k_1k_2)$$
is an isomorphism of $(H\times K,H\times K)$-bisets.
\end{proof}

\subsection{The Yoneda-Dress construction.}

Let $F$ be a biset functor and $G$ be a finite group.

\begin{deff}
\textit{The Yoneda-Dress construction of $F$ in $G$}, or \textit{$F$ shifted by $G$}, is defined as the functor $F_G$ whose value at a group $H$ is $F_G(H)=F(H\times G)$, and  $F_G(\alpha)= F(\alpha \times G)$ for any $\alpha \in kB(K,H)$.
\end{deff}

It is straightforward that shifting by $G$ defines an endofunctor of $\mathcal{F}_{\mathcal{C},k}$ (see Bouc \cite[Section 8.2]{SBb2}). If $A$ is a Green biset functor, $A_G$ is again a Green biset functor with product given by 
$$a\times^d b= A(_{K\times H\times G^{\Delta}}K\times G\times H\times G_{K\times G\times H\times G})(a\times b)$$ 
for $a\in A(K\times G)$ and $b\in A(H\times G)$, where the product on the right-hand side is the product of $A$. Then the composition in $\mathcal{P}_{A_G}$ is given by 
$$\beta\circ  \alpha=A(_{L\times H\times G^{\Delta}}L\times \overleftarrow{K} \times H\times G\times G_{L\times K \times G\times K \times H\times G})(\beta\times \alpha)$$ 
for $\beta\in A_G(L\times K)$ and $\alpha\in A_G(K\times H)$, where again the product on the right-hand side is that of $A$.

If $M$ is an $A$-module, then $M_G$ is both an $A$-module and an $A_G$-module. By the Yoneda lemma, the functors $A_G$ are projective objects in $A-\mathcal{M}od$. If $[grp]$ is a set of representatives of isomorphism classes of finite groups, the functors $A_G$ with $G\in [grp]$ form a set of projective generators of $A-\mathcal{M}od$, so the module $\bigoplus_{G\in [grp]}A_G$ is a progenerator. The following result can be found as part 3 of Bouc \cite[8.6.1]{SBb2} without a proof, thus we prove it here.

\begin{prop} \thlabel{PP}
$A-\mathcal{M}od$ has enough projectives.
\end{prop}
\begin{proof}
Let $G$ be a finite group. We will prove first that the representable functor $A_G$ is a projective object in $A-\mathcal{M}od$. Let
$$0\longrightarrow N \longrightarrow M \longrightarrow T \longrightarrow 0$$
be an exact sequence in $A-\mathcal{M}od$. There is a commutative diagram in $k-Mod$
$$\xymatrix{
0\ar[r] &A-\mathcal{M}od(A_G,N)\ar[r]\ar[d] &A-\mathcal{M}od(A_G,M)\ar[r]\ar[d] &A-\mathcal{M}od(A_G,T)\ar[r]\ar[d] &0\\
0\ar[r] &N(G)\ar[r] &M(G)\ar[r] &T(G)\ar[r] &0
}$$
where the vertical arrows are the natural isomorphisms given by the Yoneda lemma, and the bottom row is exact since evaluation on $G$ is, thus the top row is exact too and the functor $A-\mathcal{M}od(A_G,\_)$ is exact, implying $A_G$ is a projective object in $A-\mathcal{M}od$.

Now let $M$ be an $A$-module. Then there is a morphism of $A$-modules
$$\psi: \bigoplus_{G \in [grp]} \bigoplus_{s \in M(G)} A_G \longrightarrow M$$
defined by
$$\psi_H: \bigoplus_{G \in [grp]} \bigoplus_{s \in M(G)} A(H\times G) \longrightarrow M(H)$$
$$(a_{G,s})  \mapsto \sum_{(G,s)} M(a_{G,s})(s)$$
for any group $H$. Since for any $H$ there is an element $G$ in $[grp]$ such that $H \cong G$, it follows that any $\psi_H$ is an epimorphism, and so is $\psi$.
\end{proof}

The following lemma for the shifted product will be useful in the following sections.

\begin{lemma}\thlabel{shiftCOMP}
Let $A$ be a Green biset functor and $G$ be a finite group. For any finite groups $K$ and $H$ finite groups and $a\in A(K)$ and $b\in A(H\times G)$, we have
$$(Inf^{K\times G}_Ka) \times^d b=a\times b.$$
\end{lemma}
\begin{proof}
The application
$$(_{K\times H \times G^{\Delta}}K\times G\times H\times G_{K\times G\times H\times G})\circ (Inf^{K\times G}_K\times H\times G) \longrightarrow K\times H\times G$$
$$[(k_1,g_1,h_1,g_2),(k_2,h_2,g_3)]\mapsto (k_1k_2,h_1h_2,g_2g_3)$$
is a well-defined isomorphism of $(K\times H\times G,K\times H\times G)$-bisets. Then
\begin{align*}
(Inf^{K\times G}_Ka) \times^d b &=A(_{K\times H \times G^{\Delta}}K\times G\times H\times G_{K\times G\times H\times G})((Inf^{K\times G}_Ka) \times b)\\
&=A((_{K\times H \times G^{\Delta}}K\times G\times H\times G_{K\times G\times H\times G})\circ (Inf^{K\times G}_K \times H\times G))(a\times b)\\
& =A(K\times H\times G)(a\times b)\\
& =a\times b.
\end{align*}
\end{proof}

\section{$(\F,G)$-rhetorical biset functors.}

In this section $k$ and $\F$ are fields of characteristic zero. Let $\E/\F$ be an extension of fields and $H$ be a finite group. For a finitely generated $\F H$-module $V$, the notation $^{\E}V$ stands for the $\E H$-module obtained from $V$ by extension of scalars from $\F$ to $\E$. If $W$ is another finitely generated $\F H$-module, then we have isomorphisms of $\E H$-modules

\begin{itemize}
\item $^{\E}(V\oplus W)\cong {^{\E}V}\oplus {^{\E}W}$.
\item $^{\E}(V\otimes_{\F} W) \cong {^{\E}V}\otimes_{\E} {^{\E}W}$.
\item $^{\E} \F\cong \E$.
\end{itemize}

Hence we have a ring homomorphism
$$^{\E}\eta_H:R_{\F}(H)\longrightarrow R_{\E}(H)$$
sending the isomorphism class of the $\F H$-module $V$ to the class of $^{\E}V$, which can be extended to a $k$-algebra homomorphism
$$^{\E}\eta_H:kR_{\F}(H)\longrightarrow kR_{\E}(H)$$
that we denote the same way.

\begin{prop}\thlabel{PROP1}
Let $\E/\F$ be an extension of fields and $A$ and $B$ be $\F$-algebras. If $V$ is an $A$-module and $T$ is a $(B,A)$-bimodule, then we have
$$^{\E} (T \otimes_A V) \cong{^{\E}T}\otimes_{^{\E}A} {^{\E}V}$$
as $^{\E}B$-modules.
\end{prop}
\begin{proof}
Let $M$ be an $^{\E}B$-module and $\theta:{^{\E}T}\times {^{\E}V}\longrightarrow M$ be an $^{\E}A$-balanced application which is $^{\E}B$-linear in the first variable. We can put $\theta$ into the following diagram
\[
\begin{tikzcd}
T\times V \arrow[d, "\otimes"'] \arrow[r,hookrightarrow] \arrow[rr,bend left,"\theta'"]  &
  ^{\E}T\times {^{\E}V} \arrow[d,"\gamma"'] \arrow[r,"\theta"]  &
  M  \\
T\otimes_A V \arrow[r,hookrightarrow] \arrow[rru,dashed, bend right=60, "\Theta'"']  & 
^{\E}(T\otimes_A V) \arrow[ru, dashed, "\Theta"']
\end{tikzcd}
\]
where: 
\begin{itemize}
\item $\theta'$ is the composite of the natural inclusion $T\times V \hookrightarrow {^{\E}T}\times{^{\E}V}$ with $\theta$, so $\theta'$ is an $A$-balanced application which is $B$-linear in the first variable.
\item $\Theta'$ is the unique homomorphism of $B$-modules such that $\theta' = \Theta'\circ \otimes$.
\item $\Theta$ is the unique homomorphism of $^{\E}B$-modules extending $\Theta'$ to $^{\E}(T\otimes_A V)$.
\item $\gamma$ is defined by the rule $(e_1 \otimes t, e_2 \otimes v)\mapsto (e_1 e_2)\otimes (t\otimes v)$ for $t \in T$, $v \in V$ and $e_1,e_2 \in \E$, which is easily checked to be a $^{\E}A$-balanced application which is $^{\E}B$-linear in the first variable.
\end{itemize}

It is easy to see that all the triangles and the square in this diagram commute, implying the uniqueness of $\Theta$ in the sense that $\theta = \Theta \circ \gamma $. Thus $^{\E}(T\otimes_A V) \cong{^{\E}T} \otimes_{^{\E}A} {^{\E}V}$.
\end{proof}

By \thref{PROP1} it follows that for any biset $_KX_H$ and any finitely generated $\F H$-module $V$, we have
$$^{\E}(\F X\otimes_{\F H} V) \cong \E X \otimes_{\E H} {^{\E}V}$$
as $\E K$-modules, thus the morphisms $^{\E}\eta_H$ define a morphism of biset functors
$$^{\E}\eta: kR_{\F}\longrightarrow kR_{\E}.$$

If $K$ is another finite group and $U$ is a finitely generated $\F K$-module, then 
$$^{\E}(V\otimes_{\F} U) \cong {^{\E}V}\otimes_{\E} {^{\E}U}$$ 
as $\E[H\times K]$-modules, and so $^{\E}\eta$ is a morphism of Green biset functors that we call the \textit{$\E$-extension}. By Curtis \& Reiner \cite[29.7]{CR2}, $^{\E}\eta$ is a monomorphism. Thus we have the following triangle
$$\xymatrix{
kB\ar[rr]^{\lambda}\ar[dr]_{\lambda} &&kR_{\E}\\
&\RF\ar[ur]_{^{\E}\eta}
}$$
which is easily checked to be commutative. In particular if $\F=\Q$, Artin's induction theorem implies that the arrow $\lambda: kB\longrightarrow \RQ$ is an epimorphism, and so the image of $\lambda: kB\longrightarrow kR_{\E}$ is isomorphic to $\RQ$. Thus, rhetorical biset functors are precisely the modules over $\RQ$. This motivates the following definition.

\begin{deff}
If $G$ is a finite group, an \textit{$(\F,G)$-rhetorical biset functor} is a module over the functor $\RFG$. An \textit{$\F$-rhetorical biset functor} is a module over the functor $\RF$. A \textit{$G$-rhetorical biset functor} is just a $(\Q,G)$-rhetorical biset functor.
\end{deff}

Since $kR_{\F,1}\cong \RF$ as Green biset functors, $\F$-rhetorical biset functors are precisely the $(\F,1)$-rhetorical biset functors, while $\Q$-rhetorical biset functors correspond to rhetorical biset functors.

\subsection{The lattice of ideals of $\RFG$.}

We recall some results on $\F$-characters over a finite group $G$. If $n$ is a positive integer and $\omega$ is a primitive $n$-th root of $1$, then the Galois group $Gal(\F(\omega)/\F)$ of the extension $\F(\omega)/\F$ can be embedded in $(\mathbb{Z}/n\mathbb{Z})^{\times}$ as a subgroup $F_n$ since any automorphism of $\F (\omega)$ fixing $\F$ is determined by $\omega \mapsto \omega^r$ for some $r$ relatively prime to $n$. 

Recall that for a finite group $G$, its \textit{exponent} is the smallest positive integer $n$ such that $g^n=1$ for any $g\in G$, and we often denote it by $e(G)$. By Curtis \& Reiner \cite[42.4]{CR2}, if $\chi$ is an $\F$-character of $G$, then $\chi(^xg)=\chi(g^i)$ for any $g,x\in G$ and any $[i]\in F_n$.

\begin{deff}
Let $n$ be the exponent of $G$. Two elements $g,h\in G$ are said to be \textit{$\F$-conjugate} (and we write $g\sim_{\F}h$) if there are $[i]\in F_n$ and $x\in G$ such that $^xg=h^i$.
\end{deff}

It follows that to be $\F$-conjugate defines an equivalence relation on $G$. We say that a function $\xi: G\longrightarrow \F$ is an \textit{$\F$-class function} on $G$ if its restrictions to $\F$-conjugacy classes are constant functions. The space of $\F$-class functions on $G$ is an $\F$-algebra of dimension equal to the number of $\F$-conjugacy classes on $G$. Then $\F$-characters are $\F$-class functions, and by Curtis \& Reiner \cite[9.20]{CR1}, the $\F$-characters afforded by simple $\F G$-modules (or \textit{irreducible $\F$-characters} for short) are linearly independent over $\F$. By Curtis \& Reiner \cite[42.8]{CR2}, the number of irreducible $\F$-characters is the same as the number of $\F$-conjugacy classes, hence they are a basis for the space of $\F$-class functions, that we identify with the algebra $\FR(G)$.

From part 2 of Bouc \cite[7.1.3]{SBb2}, for any finite groups $H$ and $G$, any finite biset $_HU_G$ and any $\F$-character $\chi$ of $G$, we have
\begin{equation}\label{E1}
\begin{gathered}
R_{\F}(U)(\chi)(h)=\frac{1}{|G|}\sum_{\substack{u\in U, g \in G\\
hu=ug}} \chi(g)
\end{gathered}
\end{equation}
for any $h\in H$. Equation \ref{E1} holds too for $\RF$ since $\RF(U)=Id_k\otimes R_{\F}(U)$.

\begin{lemma} \thlabel{projclass}
Let $N\unlhd G$ and $\pi:G\rightarrow G/N$ be the canonical projection. If $C$ is an $\F$-conjugacy class of $G$, then $\pi(C)$ is an $\F$-conjugacy class in $G/N$.
\end{lemma}
\begin{proof}
The exponent $t$ of $G/N$ is a divisor of the exponent $n$ of $G$ so if $[i] \in F_n$ then $i$ is coprime to $t$. If $\omega$ is an $n$-th primitive root of $1$, then $\omega^{\frac{n}{t}}$ is a $t$-th primitive root of $1$ and $\omega^{\frac{n}{t}} \mapsto (\omega^{\frac{n}{t}})^i$ defines an element of $Gal(\F(\omega^{\frac{n}{t}})/\F)$ as it is the restriction of $\omega \mapsto \omega^i$, hence $[i]\in F_t$. If $g_0,g_1 \in C$, there exist $x\in G$ and $[i]\in F_n$ such that $^xg_0=g_1^i$, and so we have $^{xN}g_0N=(g_1N)^i$, thus $g_0N\sim_{\F}g_1N$. If we take any other element $gN$ which is $\F$-conjugate to $g_0N$, then there exist $xN \in G/N$ and $[i]\in F_t$ such that $(^xg_0)N=g^iN$, thus $^x(g_0^j)=gy$ for $[j]=[i]^{-1}$ in $F_n$ and some $y\in N$. Hence $gy \sim_{\F} g_0$ and $\pi(gy)=gN$.
\end{proof}

Let $C$ be an $\F$-conjugacy class on $G$. We define the $\F$-class function $e^G_C$ as
\[
e^G_C(g) = 
     \begin{cases}
       \text{1,} &\quad\text{if $g\in C$}\\
       \text{0,} &\quad\text{if $g\notin C$} \\ 
     \end{cases}
\]
for $g\in G$, which is a primitive idempotent of $\F R_{\F}(G)$. These idempotents form a full set of orthogonal primitive idempotents and they are also a basis for $\F R_{\F}(G)$. Now we will see how biset operations affect the primitive idempotents via the functor $\FRG$. For any $H$, let $\pi_2$ denote the natural projection from $H\times G$ to $G$.

\begin{lemma} \thlabel{EFFIDS1}
Let $H$ and $K$ be finite groups. If $C$ is an $\F$-conjugacy class of $H\times G$ and $\alpha \in \F B(K,H)$, we can write
\begin{equation}\label{E2}
\begin{gathered}
\FRG(\alpha)(e^{H\times G}_C)=\sum_{D\in c_{\F}(K\times G)}\lambda_D e^{K\times G}_D \in \FRG(K)
\end{gathered}
\end{equation}
for some $\lambda_D \in \F$. Then $\lambda_D\neq 0$ implies $\pi_2(D)=\pi_2(C)$.
\end{lemma}
\begin{proof}
Let $_KU_H$ be a biset. For any $(k,g)\in K\times G$, we have
$$\FRG(U)(e^{H\times G}_C)(k,g)=\FR(U\times G)(e^{H\times G}_C)(k,g)=\frac{1}{|H||G|}\sum_{\substack{(u,g_0)\in U\times G, (h,g_1) \in H\times G\\
(ku,gg_o)=(uh,g_0g_1)}} e^{H\times G}_C(h,g_1),$$
so if $\lambda_D\neq 0$ and $(k,g)\in D$, then $\FRG(U)(e^{H\times G}_C)(k,g)\neq 0$ and so there exist $(h,g_1)\in C$ and $(u,g_0)\in U\times G$ such that $(ku,gg_o)=(uh,g_0g_1)$. Then $g=\;^{g_0}g_1$, thus $\pi_2(D)=\pi_2(C)$.
\end{proof}

Let $G$ be a finite group, $n$ be its exponent and $\omega$ be an $n$-th primitive root of $1$. Since the map $Gal(\F[\omega]/\F)\longrightarrow Gal(\Q[\omega]/\F\cap \Q[\omega])$ defined by $\sigma\mapsto \sigma|_{\Q[\omega]}$ for $\sigma \in Gal(\F[\omega]/\F)$ is an isomorphism, both Galois groups are identified with the same subgroup $F_n$ of $(\Z/n\Z)^{\times}$, hence the number of $\F$-conjugacy classes and the number $\F\cap \Q[\omega]$-conjugacy classes are the same, and $R_{\F}(G)$ and $R_{\F\cap \Q[\omega]}(G)$ have the same rank. Thus 
$$^{\F}\eta_G:(\F\cap \Q[\omega])R_{\F\cap \Q[\omega]}(G) \longrightarrow (\F\cap \Q[\omega])R_{\F}(G)$$ 
is an isomorphism, and since $(\F\cap \Q[\omega])R_{\F\cap \Q[\omega]}(G)$ is split semisimple, so is $(\F\cap \Q[\omega])R_{\F}(G)$. From Curtis \& Reiner \cite[29.21]{CR2}, any extension of $\F\cap \Q[\omega]$ is a splitting field for $(\F\cap \Q[\omega])R_{\F}(G)$, in particular the fields $k[\omega]$ and the algebraic closure $\overline{k}$ of $k$, and so the algebras $k[\omega]R_{\F}(G)$ and $\overline{k}R_{\F}(G)$ are split semisimple. By Curtis \& Reiner \cite[7.2]{CR1}, the $k$-algebra $\RF(G)$ is semisimple because $k[\omega]R_{\F}(G)$ is split semisimple, and any extension of $k[\omega]$ is a splitting field for it. 

Since $\RF(G)$ is commutative, its Wedderburn factors are finite field extensions of $k$, so any ideal of it is completely determined by a set of primitive idempotents. We will use this to give a characterization of the ideals of the functor $\RFG$.

Let $E/k$ be a Galois extension and $A$ be a finite dimensional $k$-algebra. If $\sigma \in Gal (E/k)$, it induces a ring automorphism
$$\sigma \otimes 1:\; ^EA \longrightarrow\; ^EA$$
and so it sends idempotents to idempotents. Then, if $M$ is an $^EA$-module, the $^EA$-module $^{\sigma}M$ is defined to be the restriction of scalars of $M$ via $\sigma\otimes 1$, that we call the \textit{conjugate of $M$ by $\sigma$}. Conjugation by $\sigma$ sends simple modules to simple modules, and $Gal(E/k)$ acts on a set $Irr(^EA)$ of representatives of simple $^EA$-modules. By Curtis \& Reiner \cite[7.17]{CR1}, if $^EA$ is a semisimple $E$-algebra, $S$ is a simple left $^EA$-module and $e_S$ is the central primitive idempotent of $^EA$ which acts as the identity on $S$, then $(\sigma\otimes 1)(e_S)$ is a central primitive idempotent which acts as the identity on $^{\sigma}S$, for any $\sigma \in Gal (E/k)$.

So we have an action of $Gal (E/k)$ on the set of central primitive idempotents of $^EA$, and we write $\sigma \cdot e$ instead of $(\sigma\otimes 1)(e)$ for $\sigma \in Gal(E/k)$ and $e$ a central primitive idempotent of $^EA$. By Curtis \& Reiner \cite[7.18]{CR2}, if $V$ is a simple left $A$-module, then $^EV$ is a semisimple $^EA$-module, and  it is a direct sum of conjugates of one simple $^EA$-module $S$. Moreover, all conjugates of $S$ appear in $^EV$ with the same multiplicity and if $e_V \in A$ is the central primitive idempotent acting as unity on $V$, then
$$e_V=\sum_{S_i}e_{S_i}$$
where $S_i$ runs over the conjugates of $S$ and $e_{S_i}$ is the central primitive idempotent in $^EA$ acting as the identity on $S_i$.

Let $n$ be the exponent of $G$ and $\omega$ be a primitive $n$-th root of 1. Since $k[\omega]$ is a Galois extension of $k$, any primitive idempotent of $\RF(G)$ can be written as the sum of the orbit of a primitive idempotent of $k[\omega]R_{\F}(G)$ under the action of $Gal(k[\omega]/k)$. We identify $Gal(k[\omega]/k)$ with a subgroup $K_n$ of $(\mathbb{Z}/n\mathbb{Z})^{\times}$ as we did previously for $F_n=Gal(\F[\omega]/\F)$.

If $g\in G$, the idempotent $e^G_g\in k[\omega]R_{\F[\omega]}(G)$ can be written as
$$e^G_g=\frac{1}{|C_G(g)|}\sum_{\chi\in Irr_{\F[\omega]}(G)}\chi(g^{-1}) \chi$$
in terms of the basis $Irr_{\F[\omega]}(G)$ of irreducible $\F[\omega]$-characters. For any $\F$-conjugacy class $C$ of $G$, the primitive idempotent $e^G_C$ of $\RF(G)$ is an idempotent in $k[\omega]R_{\F[\omega]}(G)$, so it is a sum of some $e^G_g$ without repetitions. Since $e^G_C(g)\neq 0$ if and only if $g\in C$, then
$$e^G_C=\sum_{g\in [G\backslash C]}e^G_g $$
in $k[\omega]R_{\F[\omega]}(G)$, where $[G\backslash C]$ is a set of representatives of conjugacy classes of elements in $C$.

\begin{lemma} \thlabel{Knaction}
Let $G$ be a finite group and $n$ be its exponent. If $C$ is an $\F$-conjugacy class in $G$ and $j$ is an integer relatively prime to $n$, then $C^j=\{x^j| x\in C\}$ is also an $\F$-conjugacy class in $G$.
\end{lemma}
\begin{proof}
It is easy to see that $C^j$ is contained in some $\F$-conjugacy class $D$. On the other hand, if $z \in D$ and $x\in C$, then there exist $[i]\in F_n$ and $g\in G$ such that $^gx^j=z^i$, so $^gx=(z^t)^i$ for $[t]=[j]^{-1} \in (\mathbb{Z}/n\mathbb{Z})^{\times}$, so $z^t\in C$ and $z\in C^j$. 
\end{proof}

This lemma gives an action of $(\Z/n\Z)^{\times}$ on $c_{\F}(G)$, and by restriction, an action of $K_n$. Let $\Omega(k,\F,G)$ be the orbit space of the action of $K_n$ on $c_{\F}(G)$, and for $C\in c_{\F}(G)$, let $\mathcal{O}(C)$ denote the orbit of $C$. 

\begin{lemma} \thlabel{ORBITIDS}
Let $C$ be an $\F$-conjugacy class in $G$ and $[j] \in K_n$. Then $[j]\cdot e^G_C=e^G_{C^j}$ in $k[\omega]R_{\F}(G)$.
\end{lemma}
\begin{proof}
We have that $[j]$ induces automorphisms $[j]\otimes 1: k[\omega]R_{\F}(G) \longrightarrow k[\omega]R_{\F}(G)$ and $[j]\otimes 1: k[\omega]R_{\F[\omega]}(G) \longrightarrow k[\omega]R_{\F[\omega]}(G)$ where the second one extends the first one naturally. Then
\begin{align*}
[j] \cdot e_C &=\sum_{g\in [G\backslash C]}[j]\cdot e_g = \sum_{g\in [G\backslash C]}\frac{1}{|C_G(g)|}\sum_{\chi\in Irr_{\F[\omega]}(G)}[j](\chi(g^{-1})) \chi\\
&=\sum_{g\in [G\backslash C]}\frac{1}{|C_G(g^j)|}\sum_{\chi\in Irr_{\F[\omega]}(G)}\chi((g^j)^{-1}) \chi=\sum_{g\in [G\backslash C]}e_{g^j} =e_{C^j},
\end{align*}
where $C_G(g)=C_G(g^j)$ since $j$ is a unity modulo $n$.
\end{proof}

\begin{lemma} \thlabel{IDSKRF}
Let $e$ be a primitive idempotent of $\RF(G)$. Then there is an $\F$-conjungacy class $C$ of $G$ such that
$$e=\sum_{D\in \mathcal{O}(C)} e_D^G$$
where $\mathcal{O}(C)$ denotes the orbit of $C$ under the action of $K_n$.
\end{lemma}
\begin{proof}
Since $k[\omega]$ is a Galois extension of $k$, then $e$ can be written as the sum of the orbit of the primitive idempotent $e^G_C$ under the action of $K_n$. By \thref{ORBITIDS}, such orbit is precisely the set of idempotents $e^G_{C^i}$ with $[i]\in K_n$.
\end{proof}

From now, we write $e^G_{\mathcal{O}(C)}$ for the primitive idempotent $\sum_{D\in \mathcal{O}(C)} e_D^G$ of $\RF(G)$.

\begin{prop} \thlabel{MORPHISM}
Let $K$ and $H$ be finite groups, $C$ be an $\F$-conjugacy class of $H\times G$ and $\alpha\in kB(K, H)$. If $D$ is an $\F$-conjugacy class of $K\times G$ such that $0\neq e^{K\times G}_{\mathcal{O}(D)}\RFG(\alpha)(e^{H\times G}_{\mathcal{O}(C)})$, then $\mathcal{O}(\pi_2(D))=\mathcal{O}(\pi_2(C))$.
\end{prop}
\begin{proof}
Let $s$ be the exponent of $K\times G$ and $t$ be the exponent of $H\times G$. If $0\neq e^{K\times G}_{\mathcal{O}(D)}\RFG(\alpha)(e^{H\times G}_{\mathcal{O}(C)})$, then there are $(k,g) \in D$ and $[i]\in K_s$ such that 
$$0\neq \RFG(\alpha)(e^{H\times G}_{\mathcal{O}(C)})(k^i,g^i)=\sum_{E\in \mathcal{O}(C)}\overline{k}R_{\F,G}(\alpha)(e^{H\times G}_E)(k^i,g^i)$$
so there is $[j]\in K_t$ such that $\overline{k}R_{\F,G}(\alpha)(e^{H\times G}_{C^j})(k^{i},g^{i})\neq 0$. By \thref{EFFIDS1}, $\pi_2(D^i)= \pi_2(C^j)$ and so $\mathcal{O}(\pi_2(D))=\mathcal{O}(C)$.
\end{proof}

If $I$ is an ideal of $\RFG$, we know that $I(H)$ is an ideal of $\RF(H\times G)$ for any $H$, so it is completely determined by a unique subset $\mathcal{E}(H)$ of $\Omega(k,\F,H\times G)$.

\begin{prop} \thlabel{SETIDEAL}
Let $I$ be an ideal of $\RFG$ and $\mathcal{E}_I=\{\mathcal{O}(C)\in \Omega(k,\F,G)|e^G_{\mathcal{O}(1\times C)} \in I(1)\}$. Then
$$\mathcal{E}(H)=\{\mathcal{O}(D)\in \Omega(k,\F,H\times G)|\mathcal{O}(\pi_2(D))\in \mathcal{E}_I\}$$
for any finite group $H$.
\end{prop}
\begin{proof}
If $\mathcal{O}(C)\in \mathcal{E}(H)$ is such that $e^{1\times G}_{\mathcal{O}(E)}Def^{H\times G}_{1\times G}e^{H\times G}_{\mathcal{O}(C)}\neq 0$, then by \thref{MORPHISM}, $\mathcal{O}(\pi_2(E))=\mathcal{O}(\pi_2(C))$ and $\mathcal{O}(E)=\mathcal{O}(1\times \pi_2(C))$, hence $\mathcal{O}(\pi_2(C))\in \mathcal{E}_I$. Conversely, if $\mathcal{O}(C) \in \Omega(k,\F,H\times G)$ is such that $\mathcal{O}(\pi_2(C))\in \mathcal{E}_I$, then 
$$(e^{H\times G}_{\mathcal{O}(C)}Inf^{H\times G}_Ge^{G}_{\mathcal{O}(\pi_2 (C))})(h,g)=e^G_{\mathcal{O}(\pi_2 (C))}(g)\neq 0$$ 
for any $(h,g)\in C$, thus $e^{H\times G}_{\mathcal{O}(C)}\in I(H)$ and $\mathcal{O}(C)\in \mathcal{E}(H)$.
\end{proof}

Let $\mathcal{E}\subset \Omega(k,\F,G)$ and define
$$I_{\mathcal{E}}(H):=
\sum_{\substack{\mathcal{O}(D) \in \Omega(k,\F,H\times G)\\
\mathcal{O}(\pi_2(D))\in \mathcal{E}}}
\RF(H\times G) e^{H\times G}_{\mathcal{O}(D)}$$
for any group $H$. So we have that $I_{\mathcal{E}}(H)$ is an ideal of $\RF(H\times G)$. Then for any $\alpha \in k B(K,H)$, $\RFG (\alpha)$ restricts to a $k$-linear transformation
$$I_{\mathcal{E}}(\alpha): I_{\mathcal{E}}(H)\longrightarrow I_{\mathcal{E}}(K)$$
as a consequence of \thref{MORPHISM}. So we have

\begin{prop} \thlabel{set2ideal}
Let $\mathcal{E} \subset \Omega(k,\F,G)$. Then the assignments
$$H\mapsto I_{\mathcal{E}}(H),$$
$$\alpha \mapsto I_{\mathcal{E}}(\alpha),$$
for any $H$ and $K$ finite groups and $\alpha \in kB(K,H)$, define an ideal $I_{\mathcal{E}}$ of $\RFG$.
\end{prop}
\begin{proof}
By \thref{MORPHISM}, $I_{\mathcal{E}}$ is a biset subfunctor of $\RFG$ such that for any $H$, $I_{\mathcal{E}}(H)$ is an ideal of $\RFG(H)$, so $I_{\mathcal{E}}$ is an ideal of $\RFG$.
\end{proof}

Let $\mathfrak{I}_{k,\F,G}$ be the set of ideals of $\RFG$.

\begin{theorem} \thlabel{IDEALS}
Let $k$ and $\F$ be fields of characteristic $0$ and $G$ be a finite group. Then the function
$$\mathfrak{I}_{k,\F,G} \longrightarrow 2^{\Omega(k,\F,G)} : I \mapsto \mathcal{E}_I$$
is an isomorphism of lattices, with inverse given by $\mathcal{E}\mapsto I_{\mathcal{E}}$.
\end{theorem}
\begin{proof}
By \thref{SETIDEAL} and \thref{set2ideal}, the applications just defined are mutual inverse isomorphisms of orders.
\end{proof}

If $\mathcal{E}=\{ \mathcal{O}(C) \}$, we write $I_{\mathcal{O}(C)}$ instead of $I_{\mathcal{E}}$. \thref{IDEALS} implies that $I_{\mathcal{O}(C)}$ is a simple $(\F,G)$-rhetorical biset functor. Furthermore, since $I_{\mathcal{O}(C)}(1)=e^{1\times G}_{\mathcal{O}(1\times C)}\RF(1\times G)\cong e^G_{\mathcal{O}(C)}\RF(G)$, it follows that $I_{\mathcal{O}(C)}\cong S_{1, e^{G}_{\mathcal{O}(C)}\RF (G)}$ by \thref{thNadia}.

\begin{cor}\thlabel{SS}
The functor $\RFG$ is a semisimple object of $\RFG- \mathcal{M}od$. Furthermore,
$$\RFG \cong \bigoplus_{\mathcal{O}(C) \in \Omega(k,\F,G)}I_{\mathcal{O}(C)}$$
as $(\F,G)$-rhetorical biset functors. In particular, $\RF$ is a simple $\F$-rhetorical biset functor.
\end{cor}
\begin{proof}
$$\RFG = I_{\Omega(k,\F,G)}=\sum_{\mathcal{O}(C)\in \Omega(k,\F,G)}I_{\mathcal{O}(C)} \cong \bigoplus_{\mathcal{O}(C) \in \Omega(k,\F,G)}I_{\mathcal{O}(C)}$$ 
since the $I_{\mathcal{O}(C)}$ are non-isomorphic simple modules.
\end{proof}

When $k=\F$, then $K_n=F_n$, so we have $\mathcal{O}(C)=\{C\}$ for any $C\in c_{\F}(G)$, and in this case we write $I_C$ instead of $I_{\mathcal{O}(C)}$. \thref{SS} proves then that $\FRG\cong \oplus_{C\in c_{\F}(G)} I_C$.

\subsection{On the essential algebras of $\RFG$.}

The Artin's induction theorem implies that the linearization morphism $\lambda : k B \longrightarrow \RQ$ is an epimorphism of biset functors which also a morphism of Green biset functors. Then via this morphism we see that simple $\RQ$-modules are precisely those simple biset functors annihilated by the kernel of $\lambda$. This implies the uniqueness of minimal groups for simple rhetorical biset functors, so they are parametrized by isomorphism classes of seeds $(H,V)$ where $H$ is such that $\ER(H)\neq 0$ and $V$ is a simple $k$Out$(H)$-module which is also a $\ER(H)$-module.

\begin{deff}
Let $n$ be a positive integer. A simple $k (\Z/n\Z)^{\times}$-module $V$ is said to be \textit{primitive} if anytime $d|n$ and $ ker\; \pi_{n,d}$ acts trivially on $V$, where $\pi_{n,d}:(\Z/n\Z)^{\times}\longrightarrow (\Z/d\Z)^{\times}$ is the canonical projection, then $d=n$. This is equivalent to say that if $\psi: (\Z/n\Z)^{\times}\longrightarrow GL_k(V)$ is the linear representation of $k(\Z/n\Z)^{\times}$ on $V$, then $ker\; \pi_{n,d} \subset ker\; \psi$ implies $d=n$.
\end{deff} 

Barker \cite[1.5]{LB} states that a pair $(H,V)$ is a seed for a simple rhetorical biset functor if and only if $H$ is a cyclic group and $V$ is a primitive $kOut(H)$-module. This can be seen in two steps as it is done in \cite{ROM1}: first, proving that $H$ is cyclic if $\ER(H)\neq 0$; then, proving that if $H$ is a cyclic group such that $\ER(H)\neq 0$, a simple $k$Out$(H)$-module $V$ is also a $\ER(H)$-module if and only if $V$ is primitive.

\begin{lemma}[Romero {{\cite[3.4]{ROM1}}}] \thlabel{LB1}
\begin{enumerate} 
\item Let $H$ be a non-trivial group. Then
$$I_{\RQ}(H) =
  \sum_{\substack{K \; cyclic\\
                  |K|\; proper\;divisor\; of\; |G|}}
\left\langle\mathcal{P}_{\RQ}(K,H) \circ \mathcal{P}_{\RQ}(H,K)\right\rangle .$$
\item If $H$ is non-cyclic, then $\ER(H)=0$.
\end{enumerate}
\end{lemma}

A similar statement to Part 2 of \thref{LB1} holds for $\RFG$, while a partial generalization of Part 1 holds for $\RQG$, depending on the exponent of $H$.

\begin{prop}[Romero {{\cite[3.7]{ROM1}}}] \thlabel{LB2}
Let $H$ be a group such that $\ER(H)\neq 0$ and $V$ be a simple $k$Out$(H)$-module. Then $V$ is a $\ER(H)$-module if and only if $V$ is primitive.
\end{prop} 

It follows that $\ER(H)\neq 0$ if and only if $H$ is a cyclic group for which there exist primitive $k$Out$(H)$-modules. Now we are going to determine all the cyclic groups for which the essential algebras do not vanish. It is a known result from number theory that if $n$ is a positive integer, any character $\theta : (\mathbb{Z}/n\mathbb{Z})^{\times}\longrightarrow \C^{\times}$ factors as $\theta=\theta'\circ \pi_{n,d}$ for a unique divisor $d$ of $n$ and a unique primitive character $\theta': (\mathbb{Z}/d\mathbb{Z})^{\times}\longrightarrow \C^{\times}$. So let $Prim(n)$ be the number of primitive characters of $(\mathbb{Z}/n\mathbb{Z})^{\times}$, then we have $\phi(n)=\sum_{d|n}Prim(d)$, where $\phi$ is the Euler function. Thus $Prim$ is a multiplicative arithmetic function, since $\phi$ is. Using the Möbius inversion, we have
$$Prim (p^a)=\sum_{i=1}^a \phi(p^i)\mu(p^{a-i})=\phi(p^a)-\phi(p^{a-1})$$
for any prime $p$ and any integer $a\geq 1$, where $\mu$ is the Möbius function. Now it follows easily that $Prim(p^a)=0$ if and only if $p=2$ and $a=1$. Since $Prim$ is multiplicative, $Prim(n)\neq 0$ if and only if $n\not\equiv 2\;mod\;4$.

\begin{lemma}
Let $H$ be a finite group. Then $\ER(H)\neq 0$ if and only if $H$ is cyclic and $|H|\not\equiv 2 \;mod\;4$.
\end{lemma}
\begin{proof}
By \thref{LB1}, $\ER(H)=0$ if $H$ is non-cyclic, so we can assume $H= \Z/n\Z$ for some positive integer $n$, and we identify Out$(\Z/n\Z)$ with $(\Z/n\Z)^{\times}$. If $k$ is algebraically closed, then any simple $k(\Z/n\Z)^{\times}$-module $V$ is one dimensional and its character $\chi_V:(\Z/n\Z)^{\times}\longrightarrow k^{\times}$ can be realized as a $\C$-character, so the result follows from the previous paragraph. For arbitrary $k$, if $E \slash k$ and $V$ is a simple $k(\Z/n\Z)^{\times}$-module, then $^EV \cong S_1 \oplus \cdots \oplus S_t$ as an $E(\Z/n\Z)^{\times}$-module, where each $S_i$ is a simple $E (\Z/n\Z)^{\times}$-module. Let $\psi: (\Z/n\Z)^{\times} \longrightarrow GL_k(V)$ and  $\psi_i:(\Z/n\Z)^{\times}\longrightarrow GL_E(S_i)$ be the linear representations of $k(\Z/n\Z)^{\times}$ on $V$ and $S_i$, respectively, then $ker \; \psi =\cap_i ker\; \psi_i$, and if $d|n$, then $ker\; \pi_{n,d} \subset ker\;\psi$ if and only if $ker\; \pi_{n,d} \subset ker\;\psi_i$ for any $i$. Hence if some $S_i$ is primitive, then so is $V$. Taking $E=\overline{k}$, it follows that $\ER(\Z/n\Z)\neq 0$ for any $n\not \equiv 2 \;mod\;4$, since any simple $E(\Z/n\Z)^{\times}$-module appears as a summand of the extension to $E$ of some simple $k(\Z/n\Z)^{\times}$-module. On the other hand, $\ER(\Z/n\Z)=0$ if $n\equiv 2\;mod\;4$ for $\pi_{n,n/2}$ is an isomorphism.
\end{proof}

For finite groups $H$ and $G$, let $\pi_1: H\times G\longrightarrow H$ and $\pi_2:H\times G\longrightarrow G$ be the canonical projections. If $D\leq H\times G$, let $p_1(D)=\pi_1(D)$, $p_2(D)=\pi_2(D)$, $k_1(D)=\pi_1(ker(\pi_2|_{D}))$ and $k_2(D)=\pi_2(ker(\pi_1|_D))$. Then $k_1(D)\unlhd p_1(D)\leq H$ and $k_2(D)\unlhd p_2(D)\leq G$, and by the Goursat's lemma, $p_1(D)/k_1(D)\cong p_2(D)/k_2(D)$. If $D\leq K\times H\times G$, $p_i(D)$, $k_i(D)$, $p_{i,j}(D)$ and $k_{i,j}(D)$ are defined in a similar way, for $i,j\in \{1,2,3\}$ such that $i< j$. 

\begin{lemma}[Romero {{\cite[4.8]{ROM1}}}]\thlabel{NR1}
Let K, H and G be finite groups. If $D\leq K \times H \times G$, $L_1=p_1(D)/k_1(D)$ and $L_2=p_2(D)/k_2(D)$,  then
\begin{enumerate}
\item There exist a $K\times L_1 \times G$-set $X$ and an $L_1 \times H\times G$-set $Y$ such that
$$(K\times H\times G)/D \cong X\circ Y$$
as $(K\times H\times G)$-sets, where the composition on the right-hand side is the shifted composition.
\item There exist a $K\times L_2 \times G$-set $Z$ and an $L_2 \times H\times G$-set $W$ such that
$$(K\times H\times G)/D \cong Z\circ W$$
as $(K\times H\times G)$-sets, where the composition on the right-hand side is the shifted composition.
\end{enumerate}
\end{lemma}

\begin{lemma} \thlabel{R2}
Let $G$ be a finite group. If $H$ is a finite group and either $e(G)|e(H)$ or $(|H|,|G|)=1$, then 
$$I_{\RQG} (H)=
 \sum_{\substack{K \; cyclic\\
                  |K|\; proper\; divisor\; of\; |H|}}
 \left\langle\mathcal{P}_{\RQG}(K,H) \circ \mathcal{P}_{\RQG}(H,K)\right\rangle.$$
\end{lemma}
\begin{proof}
It is enough to prove that whenever $e(G)|e(H)$ or $(|H|,|G|)=1$, any module of the form
$$\Q[(H\times K \times G)/E] \circ \Q[(K\times H\times G)/D]$$
factors through a cyclic group whose order is a proper divisor of $|H|$, for $K$ a group of order smaller than the order of $H$ and $E\leq H\times K \times G$ and $D\leq K\times H\times G$ cyclic subgroups. By the associativity of the composition, it is enough to prove that this is so for either $\Q[(H\times K \times G)/E]$ or $\Q[(K\times H\times G)/D]$, where $K$ is a group of order smaller than $|H|$.

Suppose first that $e(G)|e(H)$ and let $D \leq K\times H\times G$ be a cyclic subgroup. Since $L_1=p_1(D)/k_1(D) \cong p_{2,3}(D)/k_{2,3}(D)$ is a cyclic subquotient of $H\times G$, then $|L_1|$ divides $e(H\times G)=e(H)$ and so it is a divisor of $|H|$. But $L_1$ is also a subquotient of $K$ and so $|L_1|$ is smaller than $|H|$. Thus $L_1$ is a cyclic group whose order is a proper divisor of $|H|$. By \thref{NR1}, $\Q[(K\times H\times G)/D] \cong \Q X\circ \Q Y$ for some $K\times L_1 \times G$-set $X$ and a $L_1\times H\times G$-set $Y$, which proves the assertion for this case.

Now let $(|H|,|G|)=1$ and again let $D \leq K\times H\times G$ be a cyclic subgroup. Then $L_2$ is a cyclic subquotient of $H$ and so $|L_2|$ divides $|H|$. But $L_2=p_2(D)/k_2(D)\cong p_{1,3}(D)/k_{1,3}(D)$ is a cyclic subquotient of $K\times G$ and so $|L_2|$ divides $|K\times G|=|K||G|$, hence $|L_2|$ divides $|K|$, which is smaller than $|H|$ and so $|L_2|$ is a proper divisor of $|H|$. By \thref{NR1}, $\Q[(K\times H\times G)/D] \cong \Q Z\circ \Q W$ for some $K\times L_2 \times G$-set $Z$ and a $L_2\times H\times G$-set $W$, so the result follows.
\end{proof}

Now we will introduce a $k$-algebra homomorphism from $\ER(H)\otimes_k \RQ(G)$ to $\ERFG(H)$.

\begin{lemma} \thlabel{lemmaESSALG}
Let $L$, $H$, $K$ and $G$ be finite groups. If $X$ is an $L\times K$-set, $Y$ is a $K\times H$-set and $Z$ and $W$ are $G$-sets, then 
$$(X\times Z)\circ (Y\times W) \cong (X\circ Y)\times (Z\times W)$$
as $L\times H \times G$-sets, where on the right-hand side of the isomorphism we have the composition in $\mathcal{P}_{kB}$ and $Z\times W$ is a $G$-set with the diagonal action, while on the left-hand side we have the shifted composition.  
\end{lemma}
\begin{proof}
The map $(X\times Z)\circ (Y\times W) \longrightarrow (X\circ Y)\times (Z\times W) :[(x,z),(y,w)] \mapsto ([x,y],(z,w))$ is an isomorphism of $L\times H \times G$-sets.
\end{proof}

Let $H$ be a finite group. We have a diagram
\begin{equation}\label{D1}
\begin{gathered}
\xymatrix{
End_{\mathcal{P}_{kB}}(H)\otimes_k kB(G) \ar[rr]^-{\mu'} \ar[d]_{\pi \otimes Id_{kB(G)}} &&End_{\mathcal{P}_{kB_G}}(H)\ar[d]^{\pi}\\
\widehat{kB}(H)\otimes_k kB(G) \ar[rr]_-{\nu'} &&\widehat{kB_G}(H)
}
\end{gathered}
\end{equation}
where:
\begin{itemize}
\item The top arrow $\mu'$ is the $k$-linear transformation induced by the product 
$$\times: kB(H\times H)\times kB(G)\longrightarrow kB(H\times H\times G)$$
of $kB$. Since $End_{\mathcal{P}_{kB}}(H)\otimes_k kB(G)$ is a unitary $k$-algebra with multiplication defined by the rule $(\alpha\otimes a)(\beta\otimes b)=(\alpha\circ \beta)\otimes (ab)$, then $\mu'$ is naturally a $k$-algebra homomorphism: by \thref{lemmaESSALG} 
\begin{align*}
\mu '((\alpha\otimes a)(\beta\otimes b)) & =\mu '((\alpha \circ \beta)\otimes (ab))=(\alpha \circ \beta)\times (ab)\\
& =(\alpha\times a)\circ (\beta\times b)=\mu '(\alpha\otimes a)\circ \mu '(\beta\otimes b)
\end{align*}
for all $\alpha,\beta\in End_{\mathcal{P}_{kB}}(H)$ and all $a,b\in kB(G)$. 
\item The $\pi$ appearing on the right vertical arrow is the natural projection
$$\pi:End_{\mathcal{P}_{kB_G}}(H)\longrightarrow \widehat{kB_G}(H).$$
\item The $\pi$ appearing on the left vertical arrow is the natural projection 
$$\pi:End_{\mathcal{P}_{kB}}(H)\longrightarrow \widehat{kB}(H)$$ 
and so $\pi\otimes Id_{kB(G)}$ is a surjective $k$-algebra homomorphism.
\item The bottom arrow $\nu'$ is defined by the rule 
$$\widehat{\alpha}\otimes a \mapsto \widehat{\alpha \times a}$$
for $\alpha\in End_{\mathcal{P}_{kB}}(H)$ and $a\in kB(G)$. This is well defined since for any $\alpha=\alpha_1\circ \alpha2$ and all $a\in kB(G)$, where $\alpha_1 \in \mathcal{P}_{kB}(L,H)$, $\alpha_2 \in \mathcal{P}_{kB}(H,L)$ and $|L|<|H|$, we have
$$\alpha\times a=(\alpha_1\times 1)\circ (\alpha_2\times a) \in I_{kB_G}(H)$$
by \thref{lemmaESSALG}. The same argument which shows that $\mu'$ is a $k$-algebra homomorphism shows that $\nu'$ is too.
\end{itemize}

Diagram \ref{D1} is easily checked to be commutative. Then linearizing all the arrows, and since linearization is a morphism of Green biset functors, we get a commutative diagram
\begin{equation} \label{D2}
\begin{gathered}
\xymatrix{
End_{\PKQ}(H)\otimes_k \RQ(G) \ar[rr]^-{\mu} \ar[d]_{\pi \otimes Id_{\RQ(G)}} &&End_{\PKQG}(H)\ar[d]^{\pi}\\
\ER(H)\otimes_k \RQ(G) \ar[rr]_-{\nu} &&\ERG(H)
}
\end{gathered}
\end{equation}
where all the arrows are homomorphisms of unitary $k$-algebras.

Let $\E/\F$ be an extension of fields of characteristic zero. Since the shift of a morphism of Green biset functors is again a morphism of Green biset functors, for any finite group $G$ we have a shifted $\E$-extension $^{\E}\eta_G: kR_{\F,G}\longrightarrow kR_{\E,G}$. Then $^{\E}\eta_G$ induces a homomorphism of associative $k$-algebras with unit $^{\E}\widehat{\eta}_{G,H}: \widehat{kR_{\F,G}}(H)\longrightarrow \widehat{kR_{\E,G}}(H)$ for any finite group $H$.

\begin{prop}\thlabel{R4}
Let $H$ be a finite group. If $\ERFG(H)\neq 0$, then $\ER(H)\neq 0$.
\end{prop}
\begin{proof}
The composite
$$\xymatrix{
\ER(H)\otimes_k \RQ(H)\ar[r]^-{\nu} &\ERG(H)\ar[r]^{^{\F}\widehat{\eta}_{G,H}} &\ERFG(H)
}$$
is a homomorphism of unitary $k$-algebras. So if $\ERFG(H)\neq 0$, then $\ER(H)\neq 0$.
\end{proof}

\begin{cor} \thlabel{R6}
If $H$ is a minimal group for a simple $(\F,G)$-rhetorical biset functor, then $H$ is cyclic and $|H|\not \equiv 2\;mod\;4$. Furthermore, $H$ is unique up to group isomorphism.
\end{cor}
\begin{proof}
If $S$ is a simple $(\F,G)$-rhetorical biset functor and $H$ is a minimal group for $S$, then $0\neq S(H)$ is a simple $\ERFG(H)$-module, and so $\ERFG(H)\neq 0$. By \thref{R4}, $\ER(H)$ is non-zero, and by \thref{LB1}, $H$ is cyclic. Then any other minimal group for $S$ is cyclic of order $|H|$, thus it is isomorphic to $H$.
\end{proof}

An immediate consequence of \thref{R6} is that \thref{thNadia} applies to $\RFG$.

\subsection{A family of simple $G$-rhetorical biset functors.}

\begin{prop} \thlabel{ESSALGH}
Let $H$ be a finite group such that $|H|$ is relatively prime to $|G|$. Then 
$$\ERG(H) \cong \ER(H)\otimes_k \RQ(G)$$
as $k$-algebras. Furthermore, $\ERG(H)\neq 0$ if and only if $\ER(H)\neq 0$.
\end{prop}
\begin{proof}
We are going to prove that $\nu$ in Diagram \ref{D2} is an isomorphism. First, since any transitive $H\times H\times G$-set $(H\times H\times G)/D$ is isomorphic to a product $((H\times H)/D_1) \times (G/D_2)$, for some $D_1\leq H\times H$ and $D_2\leq G$ such that $D=D_1\times D_2$, we have that in this case the $\mu$ in Diagram \ref{D2} is surjective. By counting conjugacy classes of cyclic subgroups of $H\times H$, $G$ and $H\times H\times G$, we can see that $End_{\PKQ}(H)\otimes_k \RQ(G)$ and $End_{\PKQG}(H)$ have the same dimension over $k$, thus $\mu$ is an isomorphism. Then, since $\pi \mu$ is surjective, so is $\nu$. Now let $z \in ker\; \nu$, then $(\pi\otimes Id_{\RQ(G)})(z_0)=z$ for some $z_0 \in \ER(H)\otimes \RQ(G)$. The commutativity of Diagram \ref{D2} shows that $\mu(z_0) \in I_{\RQG}(H)$, and by \thref{R2}, 
$$\mu(z_0)=\sum_i \lambda_i\left(\Q[(H\times K_i \times G)/D_i]\circ \Q[(K_i\times H\times G)/E_i]\right)$$ 
for some cyclic groups $K_i$ whose orders are proper divisors of $|H|$, some $D_i \leq H\times K_i\times G$ and $E_i\leq K_i\times H\times G$, and coefficients $\lambda_i\in k$. Since the $|K_i|$ are divisors of $|H|$, then we have $(|H\times K_i|,|G|)=1$, and so $D_i=R_i\times S_i$ for some $R_i\leq H\times K_i$ and $S_i\leq G$, and $E_i=T_i\times U_i$ for some $T_i\leq K_i\times H$ and $U_i\leq G$. By \thref{lemmaESSALG}, 
$$\mu(z_0)=\sum_i \lambda_i\left(\Q[(H\times K_i)/R_i]\circ \Q[(K_i\times H)/T_i]\right)\times \left( \Q[G/S_i]\otimes_{\Q} \Q[G/U_i]\right)$$
and so
$$z_0=\sum_i \lambda_i\left(\Q[(H\times K_i)/R_i]\circ \Q[(K_i\times H)/T_i]\right)\otimes ( \Q[G/S_i]\otimes_{\Q} \Q[G/U_i])$$
since $\mu$ is an isomorphism. Thus $z=(\pi\otimes Id_{\RQ(G)})(z_0)=0$.
\end{proof}

Let $H$ be a cyclic group of order relatively prime to $|H|$ and such that $\ER(H)\neq 0$, and let $[CS(G)]$ be a set of representatives of the conjugacy classes of cyclic subgroups of $G$. Since $\RQ(G) \cong \prod_{C\in [CS(G)]} k$ as $k$-algebras, we have 
$$\ERG(H)\cong \ER(H) \otimes \RQ(G) \cong \prod_{C\in [CS(G)]}\ER(H)$$
as $k$-algebras, so if $V$ is a simple $\ERG(H)$-module, then it is isomorphic to the restriction of scalars of a simple $\ER(H)$-module via the projection $\pi_C$ to the $C$-th coordinate factor $\ER(H)$ for a unique $C \in [CS(G)]$. Then $(1\times e^G_C)V= V$, and $(1\times e^G_D)V=0$ for any other $C\neq D\in [CS(G)]$.

Given two triplets $(H,V,C)$ and $(K,W,D)$, where $(H,V)$ and $(K,W)$ are seeds on $\mathcal{P}_{\RQ}$ such that both $|H|$ and $|K|$ are relatively prime to $|G|$ and $C$ and $D$ are cyclic subgroups of $G$, we say that $(H,V,C)$ and $(K,W,D)$ are isomorphic if $(H,V)$ and $(K,W)$ are isomorphic as seeds on $\mathcal{P}_{\RQ}$ and $C$ and $D$ are conjugates in $G$. 

If $(H,U)$ is a seed on $\mathcal{P}_{\RQG}$ such that $(|H|,|G|)=1$, then by the previous paragraphs $U$ can be seen as the restriction of scalars of a simple $\ER(H)$-module $V$ via the $C$-th projection $\pi_C:\ERG(H)\longrightarrow \ER(H)$, so we can assign $(H,U)$ a triplet $(H,V,C)$, where $(H,V)$ is a seed on $\mathcal{P}_{\RQ}$ and $C$ is a cyclic subgroup of $G$. Then, isomorphic seeds are assigned isomorphic triplets.

Now given a triplet $(H,V,C)$ with $(|H|,|G|)=1$, we can construct a simple $G$-rhetorical biset functor $T_{H,V,C}=S_{H,\pi_C^*V}$, where $\pi_C^*V$ denotes the restriction of scalars of $V$ via $\pi_C$. It is not hard to see that $T_{K,W,D}\cong T_{H,V,C}$ if and only if $(K,W,D)$ and $(H,V,C)$ are isomorphic.

\begin{cor}
There is a bijection between the set of isomorphism classes of simple $G$-rhetorical biset functors whose minimal groups have order relatively prime to $|G|$ and the set of isomorphism classes of triplets $(H,V,C)$, where $(H,V)$ is a seed on $\mathcal{P}_{\RQ}$ such that $|H|$ is relatively prime to $|G|$ and $C$ is a cyclic subgroup of $G$.
\end{cor}
\begin{proof}
By \thref{R6}, minimal groups for simple $G$-rhetorical biset functors are unique up to group isomorphism, thus \thref{thNadia} holds for $\RQG$. Now if $S$ is a simple $G$-rhetorical biset functor whose minimal group $H$ has order relatively prime to $|G|$, then $S\cong S_{H,U}$ for a unique seed $(H,U)$ up to isomorphism, and we can assign $(H,U)$ a triplet $(H,V,C)$ where $U\cong \pi_c^*V$. We define an application from the set of isomorphism classes of simple $G$-rhetorical biset functors to the set of isomorphism classes of triplets by sending the class of $S$ to the class of $(H,V,C)$. The inverse application is given by sending the class of the triplet $(H,V,C)$ to the class of $T_{H,V,C}$.
\end{proof}

\section{The category of $\CRG$-modules.}

By Bouc \cite[7.3.5]{SBb2}, the functor $\CR$ is a semisimple biset functor with infinitely many simple factors, but it is simple as a module over itself.

\begin{prop}[Romero {{\cite[4.3]{ROM1}}}] \thlabel{NR2}
$\CR$ is a simple $\mathbb{C}R_{\mathbb{C}}$-module. In particular, it is a simple Green biset functor.
\end{prop}

The core of the proof of this proposition lies in the fact that in $\mathcal{P}_{\CR}$ the hom-sets are generated by morphisms which factor through the trivial group. We see now that this is so in the shifted case. 

\begin{prop} \thlabel{R3} 
Let $G$ be a finite group. Then for all finite groups $H$ and $K$ we have
$$\mathcal{P}_{\mathbb{C}R_{\mathbb{C},G}}(H,K)=\langle \mathcal{P}_{\mathbb{C}R_{\mathbb{C},G}}(1,K) \circ \mathcal{P}_{\mathbb{C}R_{\mathbb{C},G}}(H,1)\rangle.$$
\end{prop}
\begin{proof}
The hom-set $\mathcal{P}_{\mathbb{C}R_{\mathbb{C},G}}(H,K)=\CR(K\times H\times G)$ is generated over $\C$ by the isomorphism classes of simple $\C [K\times H\times G]$-modules. Let $S$ be a simple $\mathbb{C}[K\times H \times G]$-module, then there exist a simple $\mathbb{C}K$-module $V$ and a simple $\mathbb{C}[H\times G]$-module $W$ such that $S\cong V\otimes_{\C} W$ as $\mathbb{C}[K\times H\times G]$-modules. By \thref{L2} and \thref{shiftCOMP},
$$(Inf_K^{K\times 1\times G}V) \circ (Inf_{H\times G}^{1\times H\times G}W) \cong (Inf_K^{K\times G}V) \times^d W \cong V\otimes_{\C} W \cong S$$
thus $[S]$ factors through the trivial group.
\end{proof}

From \thref{R3} we see that $\widehat{\CRG}(H)\neq 0$ if and only if $H$ is the trivial group. Then the simple $\CRG$-modules are in one-to-one correspondence with the simple $\widehat{\mathbb{C}R_{\mathbb{C},G}}(1)$-modules, up to isomorphism. By \thref{L1}, $\widehat{\mathbb{C}R_{\mathbb{C},G}}(1) \cong \mathbb{C}R_{\mathbb{C}}(G)$ as $\C$-algebras, so the number of simple $\mathbb{C}R_{\mathbb{C},G}$-modules is the same as the number of conjugacy classes of elements in $G$. Thus the simple $\CRG$-modules are precisely the minimal ideals $I_C$ of $\CRG$ up to isomorphism, where $C$ runs over the conjugacy classes of elements in $G$. Moreover, by \thref{SS},
$$\CRG \cong \bigoplus_{C\in \mathcal{C}(G)}I_C$$
as $\CRG$-modules, where $\mathcal{C}(G)$ denotes the set of conjugacy classes of elements of $G$.

\subsection{An equivalence.}

We prove now that the evaluation at the trivial group $ev_1: \CRG-\mathcal{M}od \longrightarrow \CR(G)-Mod$ is an equivalence of categories.

\begin{prop}
There is a $\C$-linear equivalence between the category $\CRG-\mathcal{M}od$ and the category $\CR(G)-Mod$.
\end{prop}
\begin{proof}
Let $L_1$ be the left adjoint of $ev_1$ as defined in Section 2. It is easy to see that $ev_1\circ L_1 \cong 1_{\CR(G)-Mod}$. From \thref{PP}, any $\CRG$-module $M$ can be covered by a sum of representable functors, so we have an exact sequence
$$\xymatrix{
\bigoplus_i \C R_{\C, H_i\times G}\ar[r] &\bigoplus_j\C R_{\C, H_j\times G}\ar[r] &M\ar[r]  &0.
}$$
of $\CRG$-modules. Then we have a commutative diagram
\begin{equation} \label{D3}
\begin{gathered}
\xymatrix{
\bigoplus_i\C R_{\C, H_i\times G}\ar[r] &\bigoplus_j\C R_{\C, H_j\times G}\ar[r] &M\ar[r]  &0\\
\bigoplus_iL_{1,\CR(1\times H_i\times G)}\ar[r]\ar[u] &\bigoplus_jL_{1,\CR(1\times  H_j\times G)}\ar[r]\ar[u] &L_{1,M(1)}\ar[r]\ar[u]  &0
}
\end{gathered}
\end{equation}
where the vertical arrows are the component arrows of the counit $L_1\circ ev_1\longrightarrow 1_{\CRG -\mathcal{M}od}$. Since $L_1$ is right exact, the bottom row is exact. We want to prove that the arrow from $L_{1,M(1)}$ to $M$ is an isomorphism. It is enough to prove that the left-hand side and the middle vertical arrows are isomorphisms, for this would imply that the right-hand side vertical arrow is an isomorphism too. Then it is enough to prove that the counit component arrow at $\C R_{H\times G}$ is an isomorphism for any $H$. This natural transformation is given by composition
$$\xymatrix{
\CR(K\times 1\times G) \otimes_{\CR(1\times 1\times G)} \CR(1\times H\times G)\ar[r]^-{\circ} &\CR(K\times H\times G)
}$$
$$a\otimes b \mapsto a\circ b$$
for any $K$, and by \thref{R3}, this application is surjective. We will prove that this natural transformation is an isomorphism by showing that $\CR(K\times 1\times G) \otimes_{\CR(1\times 1\times G)} \CR(1\times H\times G)$ and $\CR(K\times H\times G)$ are isomorphic finite-dimentional $\C$-vector spaces. \thref{L1} shows that 
$$Inf^{1\times 1\times G}_G: \CR(G)\longrightarrow End_{\mathcal{P}_{\CRG}}(1)$$ 
is an isomorphism of algebras, and $\CR (K\times 1\times G)$ is a right $\CR(G)$-module by restrictions of scalars. On the other hand, 
$$Id_{\CR(K)}\otimes Inf^{1\times G}_G:\CR(K)\otimes_{\C} \CR(G)\longrightarrow \CR(K\times 1\times G)$$
sending $a\otimes b$ to $a\times (Inf^{1\times G}_Gb)$ is an isomorphism of $\C$-vector spaces. This application is an isomorphism of right $\CR(G)$-modules, since 
\begin{align*}
a\times (Inf^{1\times G}_G(bc)) &=(a\times (Inf^{1\times G}_Gb))\times^d(Inf^{1\times G}_Gc)\\
& =(a\times (Inf^{1\times G}_{G}b))\circ (Inf^{1\times 1\times G}_Gc)
\end{align*}
for any $a\in \CR(K)$ and $b,c\in \CR(G)$. Similarly, it can be proved that $\CR(1\times H\times G)$ and $\CR(H)\otimes_{\C}\CR(G)$ are isomorphic as left $\CR(G)$-modules. Then we have 
\begin{align*}
\CR (K\times 1\times G) &\otimes_{\CR(1\times 1\times G)} \CR(1\times H\times G)\\
&\cong \left(\CR(K)\otimes_{\C} \CR(G)\right)\otimes_{\CR(G)}\left(\CR(H)\otimes_{\C} \CR(G)\right)\\
&\cong \CR(K)\otimes_{\C} \CR(H)\otimes_{\C} \CR(G)\\
&\cong \CR(K\times H\times G).
\end{align*}

\end{proof}

\section*{Acknowledgement}

This work is part of my PhD project at the Posgrado Conjunto en Ciencias Matemáticas UNAM-UMSNH, supported by a CONACyT grant. I would like to thank to my advisors Gerardo Raggi and Nadia Romero for their commitment in directing this project. I am deeply grateful to Serge Bouc for all the enlightening conversations at LAMFA. I also want to thank to LaSoL for its financial support.

\end{document}